\documentclass[11pt]{amsart}

\usepackage{subfiles}
\usepackage{comment}
\usepackage{float}
\usepackage{marginnote}
\usepackage{tabu}
\usepackage{euscript}
\usepackage[dvipsnames]{xcolor}   		             		
\usepackage{graphicx}			
\usepackage{amssymb}
\usepackage{mathrsfs}
\usepackage{amsthm}
\usepackage{amsmath}
\usepackage{stmaryrd}
\usepackage{tikz}
\usepackage{tikz-cd}
\usepackage{accents}
\usepackage{upgreek}
\usepackage{enumerate}
\usepackage{bm}
\usepackage{mathtools}
\usepackage[all]{xy}
\usepackage{caption}
\usepackage{url}
\usepackage{float}
\usepackage{todonotes} 
\usepackage{colonequals}
\usepackage{bbm}
\usepackage{longtable}
\usepackage[full]{textcomp}
\usepackage[cal=cm]{mathalfa}
\usepackage{xparse}
\usepackage{comment}
\usepackage{cite}

\definecolor{lightyblue}{HTML}{B1B1FB}

\usetikzlibrary{calc}
\usetikzlibrary{fadings}
\usetikzlibrary{decorations.pathmorphing}
\usetikzlibrary{decorations.pathreplacing}
\usepackage{tikz,tikz-cd,tikz-3dplot}
\usepackage{pgfplots}
\usetikzlibrary{arrows,shadows,positioning, calc, decorations.markings, 
hobby,quotes,angles,decorations.pathreplacing,intersections,shapes}
\usepgflibrary{shapes.geometric}
\usetikzlibrary{fillbetween,backgrounds}

\usepackage[margin=1.2in]{geometry}

\tikzset{
  commutative diagrams/.cd, 
  arrow style=tikz, 
  diagrams={>=stealth}
}
\tikzset{
  arrow/.pic={\path[tips,every arrow/.try,->,>=#1] (0,0) -- +(0,4pt);},
  pics/arrow/.default={triangle 90}
}
\tikzset{->-/.style={decoration={
  markings,
  mark=at position .6 with {\arrow{latex}}},postaction={decorate}}
  }
\tikzset{
  c/.style={every coordinate/.try}
}

\theoremstyle{theorem}
\newtheorem{innercustomthm}{Theorem}
\newenvironment{customthm}[1]
  {\renewcommand\theinnercustomthm{#1}\innercustomthm}
  {\endinnercustomthm}
  
\theoremstyle{definition}

  \theoremstyle{definition}
\newtheorem{innercustomconj}{Question}
\newenvironment{customquestion}[1]
  {\renewcommand\theinnercustomconj{#1}\innercustomconj}
  {\endinnercustomconj}

\makeatletter
\def\@tocline#1#2#3#4#5#6#7{\relax
  \ifnum #1>\c@tocdepth 
  \else
    \par \addpenalty\@secpenalty\addvspace{#2}%
    \begingroup \hyphenpenalty\@M
    \@ifempty{#4}{%
      \@tempdima\csname r@tocindent\number#1\endcsname\relax
    }{%
      \@tempdima#4\relax
    }%
    \parindent\z@ \leftskip#3\relax \advance\leftskip\@tempdima\relax
    \rightskip\@pnumwidth plus4em \parfillskip-\@pnumwidth
    #5\leavevmode\hskip-\@tempdima
      \ifcase #1
       \or\or \hskip 1em \or \hskip 2em \else \hskip 3em \fi%
      #6\nobreak\relax
    \dotfill\hbox to\@pnumwidth{\@tocpagenum{#7}}\par
    \nobreak
    \endgroup
  \fi}
\makeatother

\newcounter{marginnote}
\DeclareMathAlphabet{\mathpzc}{OT1}{pzc}{m}{it}

\usepackage[backref=page]{hyperref}
\hypersetup{
  colorlinks   = true,          
  urlcolor     = blue,          
  linkcolor    = purple,          
  citecolor   = blue             
}

\theoremstyle{theorem}
\newtheorem{theorem}{Theorem}[section]

\newtheorem{corollary}[theorem]{Corollary}
\newtheorem{lemma}[theorem]{Lemma}
\newtheorem{proposition}[theorem]{Proposition}

\theoremstyle{definition}

\newtheorem{remark}[theorem]{Remark}

\newtheorem*{runningexample*}{Running example}

\newtheorem*{aside*}{Aside}

\newtheorem{definition}[theorem]{Definition}
\newtheorem{example}[theorem]{Example}

\newtheorem{notation}[theorem]{Notation}
\newtheorem{proposition-definition}[theorem]{Proposition-Definition}
\newtheorem{question}[theorem]{Question}

\newcommand{\rk}{\operatorname{rk}}

\DeclareMathOperator{\Hom}{Hom}

\DeclareMathOperator{\gp}{gp}
\DeclareMathOperator{\sat}{sat}
\newcommand{\suchthat}{\mathrel{:}}
\newcommand{\RR}{\mathbb{R}}

\newcommand{\sqC}{\scalebox{0.9}[1.3]{$\sqsubset$}}
\newcommand{\tropC}{\sqC}
\newcommand{\ftrop}{\mathsf{f}}

\newcommand{\Gm}{\mathbb{G}_{\operatorname{m}}}

\newcommand{\bcd}{\begin{center}\begin{tikzcd}}
\newcommand{\ecd}{\end{tikzcd}\end{center}}

\newcommand{\Aaff}{\mathbb{A}}

\newcommand{\N}{\mathbb{N}}
\newcommand{\Z}{\mathbb{Z}}
\newcommand{\Q}{\mathbb{Q}}

\newcommand{\R}{\mathbb{R}}

\newcommand{\Speck}{\operatorname{Spec}\kfield}
\newcommand{\kfield}{\Bbbk}

\newcommand{\Acal}{\mathcal{A}}

\newcommand{\Ccal}{\mathcal{C}}

\newcommand{\Log}{\mathsf{Log}}

\newcommand{\Spec}{\operatorname{Spec}}

\makeatletter
\DeclareRobustCommand{\cev}[1]{%
  \mathpalette\do@cev{#1}%
}
\newcommand{\do@cev}[2]{%
  \fix@cev{#1}{+}%
  \reflectbox{$\m@th#1\vec{\reflectbox{$\fix@cev{#1}{-}\m@th#1#2\fix@cev{#1}{+}$}}$}%
  \fix@cev{#1}{-}%
}
\newcommand{\fix@cev}[2]{%
  \ifx#1\displaystyle
    \mkern#23mu
  \else
    \ifx#1\textstyle
      \mkern#23mu
    \else
      \ifx#1\scriptstyle
        \mkern#22mu
      \else
        \mkern#22mu
      \fi
    \fi
  \fi
}

\begin{document}
 
\title{Universality for tropical and logarithmic maps}
\author{Gabriel Corrigan, Navid Nabijou, Dan Simms}
 
\begin{abstract} We prove that every toric monoid appears in a space of maps from tropical curves to an orthant. It follows that spaces of logarithmic maps to Artin fans exhibit arbitrary toric singularities: a virtual universality theorem for logarithmic maps to pairs. The target rank depends on the chosen singularity: we show that the cone over the $7$-gon never appears in a space of maps to a rank-$1$ target. We obtain similar results for tropical maps to affine space.

\end{abstract}

\maketitle
\tableofcontents

\section*{Introduction}

\noindent Mn\"ev universality --- also known as Murphy's or Vakil's law --- asserts that a given class of moduli spaces exhibits arbitrary singularities. It is known to hold in many important cases, including incidence schemes, Hilbert schemes, Chow varieties, and moduli of toric vector bundles \cite{Mnev1,Mnev2,VakilMurphy,LeeVakil,PayneModuliToric,KatzPayneRealization,ErmanHilbert,JelisiejewHilbert,BrandtJonesLeeRanganathan,UtstolNodland,APL,AP1}.

Spaces of stable maps satisfy universality \cite{VakilMurphy}. But while they can be arbitrarily singular, they are always ``virtually'' smooth: a space of stable maps admits a perfect obstruction theory \cite{BehrendFantechi,LiTian} which controls deformations and furnishes the space with many ``virtual'' structures analogous to those found on a smooth variety: fundamental class \cite{BehrendManin}, torus localisation \cite{GraberPandharipande}, Hirzebruch--Riemann--Roch theorem \cite{FantechiGoettsche} and so on. These virtual structures play a central role in Gromov--Witten theory.

Recent years have witnessed the ascendance of logarithmic Gromov--Witten theory and the associated moduli spaces of stable logarithmic maps \cite{GrossSiebertLog,ChenLog,AbramovichChenLog}. This class of moduli spaces includes ordinary stable maps, hence automatically satisfies universality.

On the virtual level, however, there is a crucial difference: the obstruction theory for the space of stable logarithmic maps is defined relative to the space of prestable logarithmic maps to the Artin fan \cite[Proposition~6.3.1]{AbramovichWiseBirational}
\[ \Log(X|D) \to \Log(\Acal_{X|D}).\]
The stack $\Log(\Acal_{X|D})$ is irreducible, so virtual pullback furnishes $\Log(X|D)$ with a virtual fundamental class \cite{ManolachePull}. However $\Log(\Acal_{X|D})$ is \emph{not} in general smooth, or even virtually smooth. As a result, $\Log(X|D)$ does not admit a perfect obstruction theory in the absolute sense. This marks a sharp departure from the theory of ordinary stable maps.

\subsection{Universality} A basic question arises: which singularities does $\Log(\Acal_{X|D})$ exhibit? Equivalently: which ``virtual singularities'' does $\Log(X|D)$ exhibit?

The question is not only of theoretical interest: virtual singularities pose novel technical difficulties --- e.g. for torus localisation \cite{GraberMFOReport} and logarithmic intersection theory \cite{MolchoRanganathan} --- so it is worthwhile investigating their complexity.

Though $\Log(\Acal_{X|D})$ is singular, it carries a logarithmic structure with respect to which it is logarithmically smooth \cite[Proposition~1.6.1]{AbramovichWiseBirational}. It follows that it has at worst toric singularities. The main result of this paper is that it exhibits \emph{arbitrary} toric singularities. This is a (virtual, toric) universality theorem for logarithmic maps.

\begin{customthm}{A}[Theorem~\ref{cor: main cor}] \label{cor: main introduction} Every toric singularity appears in a moduli space of prestable logarithmic maps $\Log(\Acal^n)$, where $\Acal=[\Aaff^{\! 1}/\Gm]$ and $n \in \N$ (depending on the singularity). This phenomenon occurs already for genus zero source curves.
\end{customthm}

The proof is tropical. Given a prestable logarithmic map to $\Acal^n$ the associated discrete data is recorded in a tropical type of map to $\R^n_+$ (Definition~\ref{def: tropical type}). This tropical type defines a toric monoid, whose corresponding affine toric variety gives the local singularity type of $\Log(\Acal^n)$ (Proposition~\ref{prop: representability implies map to Artin fan exists}). Hence Theorem~\ref{cor: main introduction} follows immediately from:

\begin{customthm}{B}[Theorem~\ref{thm: main theorem}] \label{thm: main introduction} Given a toric monoid $P$ there exists an $n \in \N$ and a representable tropical type of map to $\R^n_+$ whose associated tropical monoid is $P$. Moreover we may choose a tropical type with genus zero source curve.\end{customthm} 
To prove this, we start with an arbitrary presentation of the monoid $P$. We then modify it to a presentation which is both ``bipartite'' and ``positive'' (Definitions~\ref{def: bipartite} and \ref{def: positive}). Bipartite means that the generators can be partitioned into two sets
\[ G = G_1 \sqcup G_2 \]
such that elements of $G_1$ only appear on left-hand sides of relations and elements of $G_2$ only appear on right-hand sides. Positive simply means that none of the generators are zero.

After reducing to a bipartite and positive presentation (Proposition~\ref{prop: sanitising presentation}), we then build a tropical type whose shape mirrors the special structure of this presentation (see the proof of Theorem~\ref{thm: main theorem}). The source graph $\Gamma$ consists of two paths, whose edges are indexed by the generators belonging to the two sets $G_1,G_2$. The relations are encoded in a single continuity equation in a high-dimensional tropical target $\R^n_+$. Here $n$ depends on $P$ and is the number of relations in the chosen bipartite and positive presentation.

While we focus on tropical maps to orthants, our techniques also establish universality for tropical maps to affine spaces; see Section~\ref{sec: variants}.

\subsection{Boundedness}
As remarked, in our construction the target rank $n$ depends on the chosen monoid $P$. We do not know whether this dependence is essential:

\begin{customquestion}{C}[Question~\ref{question main}] \label{question introduction} Does there exist a single $n \in \N$ such that every toric monoid appears as the tropical monoid associated to some tropical type of map to $\R^n_+$?
\end{customquestion} 

In Section~\ref{sec: boundedness} we prove that $n=1$ does not suffice:

\begin{customthm}{D}[Theorem~\ref{thm: 7-gon inaccesible}] \label{thm: 7-gon inaccesible introduction} For $k \geqslant 7$ the cone over a $k$-gon does not appear as the tropical monoid associated to any tropical type of map to $\R_+$.
	\end{customthm}
The proof strategy is to bound both the rank and the minimal number of generators of a tropical monoid in terms of the number of vertices of the source graph. This leads to a contradiction when the rank is much smaller than the minimal number of generators, as happens for cones over polygons with large numbers of edges.
	
In the opposite direction we also show (Theorem~\ref{thm: get all 2D cones}) that every rank-$2$ monoid does appear as the monoid associated to a tropical type of map to $\R_+$. The argument is quite subtle, with saturation playing an essential role (see Example~\ref{example get 2D cone}). We hope that the ideas developed there will shed light on Question~\ref{question introduction} for $n \geqslant 2$.

\subsection{Sources and targets} In the enumerative geometry of logarithmic and tropical curves, there is a tension between two types of complexity (see \cite[Section~0.1]{BNR2}):
\begin{itemize}
\item Source complexity: genus of the source curve.
\item Target complexity: rank of the target. 
\end{itemize}
Together our results show that, at least as far as singularities of the moduli space are concerned, target complexity is the more fundamental obstruction: with trivial source genus and arbitrary target rank we obtain all toric monoids (Theorem~\ref{thm: main introduction}), but with arbitrary source genus and trivial target rank we do not (Theorem~\ref{thm: 7-gon inaccesible introduction}).

\subsection{Context} This paper forms part of a broader research direction: investigating the geometry of moduli spaces of logarithmic maps. The present paper focuses on local aspects; for recent progress on global aspects, see \cite{RanganathanSkeletons1,KannanCutPaste,KNSZ}.

Spaces of logarithmic maps form a rich class of modular toroidal embeddings, and as such their geometry is of intrinsic interest. In addition, this direction has produced novel approaches to enumerative problems \cite{MaxContacts,BNR2}.

As already remarked, virtual singularities pose unique technical challenges. In many cases --- e.g. for virtual localisation or applications of Fulton's blowup formula --- these are circumvented, either explicitly or implicitly, by passing to a resolution. Our results imply that, in principle, this step is as complicated as the full toric resolution algorithm.

It is currently unknown whether there exists a general-purpose modular desingularisation of the space of logarithmic maps. For maps to $\Acal$ such a desingularisation is provided by the space of maps to expansions, which is described tropically as the space of image-ordered maps to $\R_+$ (see e.g. \cite[Section~3.1]{BNRGenus1}). For maps to $\Acal^n$ the analogous space depends on combinatorial choices \cite{RangExpansions, MR20} and it is natural to wonder whether there is a choice which produces smooth moduli. It follows from our results that such a putative moduli space would automatically implement resolution of arbitrary toric singularities.

\subsection*{Notation} For a positive integer $n$ we set $[n]\colonequals \{1,\ldots,n\}$. We let $\R^n_+$ denote the rational polyhedral cone $\R^n_{\geqslant 0}$ with $\Z^n$ as the underlying lattice.

To declutter notation, we write $\Log(X|D)$ for a space of stable logarithmic maps and $\Log(\Acal_{X|D})$ for a space of prestable logarithmic maps to the Artin fan. Stability, however, plays no essential role: all results concerning the singularities of moduli spaces apply whether or not we impose stability.

\subsection*{Acknowledgements} We thank Luca~Battistella and Dhruv~Ranganathan for inspiring conversations. G.C. is indebted to Jacqueline Jones for her sustained dedication and support, without whom he would not have had many opportunities for which he is grateful.

This work took place during July--August 2022, as part of the ``Research in the CMS'' undergraduate summer research program at the University of Cambridge. We thank the DPMMS Bursary Fund, the Herchel Smith Fund, and Gonville and Caius College for financial support.

\section{Preliminaries}

\noindent The purpose of this section is to establish conventions and set notation. Detailed treatments of the material can be found in the following references: for monoids, \cite[Part~I]{Ogus}; for logarithmic and tropical maps, \cite[Section~2.5]{AbramovichChenGrossSiebertDegeneration}; for Artin fans, \cite[Section~5]{AbramovichEtAlSkeletons} and \cite{AbramovichWiseBirational}.

\subsection{Presentations of monoids}  A \textbf{monoid} is a set with a single binary operation, satisfying all the axioms of an abelian group except the existence of inverses. A monoid \textbf{presentation} is a pair $(G|R)$ where $G$ is a finite generating set and 
\[ R \subseteq \N^G \times \N^G \]
is a finite relation set. All monoids will be assumed finitely-presented. We often write
\[ (w_1=w_2) \in R\]
to denote the pair $(w_1,w_2) \in R$. The presentation $(G|R)$ defines a monoid $\N^G/R$ as the colimit
\[
\begin{tikzcd}
 \N^R \ar[r,yshift=3pt,"w_1" above] \ar[r,yshift=-3pt,"w_2" below] & \N^G \ar[r,dashed] & \N^G/R.
\end{tikzcd}
\] 
For a discussion of limits and colimits of monoids, see \cite[Chapter I.1.1]{Ogus}. A monoid is \textbf{toric} if it is finitely-generated, torsion-free, integral, saturated, and sharp; for a list of terminology see \cite[Appendix~A]{ChenLog}. Toric monoids are precisely those of the form
\[ P = \upsigma^\vee \cap M \]
where $M$ is a lattice and $\upsigma \subseteq  M_{\R}^\vee$ is a strictly convex rational polyhedral cone of full dimension. The association $P \leftrightarrow \upsigma$ is a duality. Every finitely-generated monoid $P$ admits a \textbf{torification}
\[ P \to P^{\operatorname{tor}} \]
which is the universal toric monoid to which $P$ maps. This construction is functorial, forming a left adjoint to the inclusion of the full subcategory of toric monoids.

\begin{definition} \label{def: toric monoid associated to presentation} Given a monoid presentation $(G|R)$, the \textbf{associated toric monoid} is the torification of $\N^G/R$.
\end{definition}

\begin{remark}\label{rmk: torification}
The precise order of operations for torification is as follows:
\begin{enumerate}
\item \textbf{Integralise and remove torsion.} This is achieved in a single step by replacing the monoid by its image in the torsion-free part of its groupification.
\item \textbf{Saturate.} It follows immediately from the definitions that this preserves integrality and torsion-freeness.
\item \textbf{Sharpen.} This preserves integrality, torsion-freeness and saturatedness by \cite[Propositions~I.1.3.3 and I.1.3.5(4)]{Ogus}.
\end{enumerate}
During this process, saturation is the only step where the minimal number of generators can increase, and sharpening is the only step where the rank of the groupification can change. These facts will be used in Section~\ref{sec: boundedness}.

Removing torsion in (1) is in fact unnecessary, since the sharpening in (3) also guarantees this. We keep the redundancy to align more closely with the existing literature; see Remark~\ref{rmk: toric monoid vs GS monoid}.
\end{remark}

\subsection{Tropical maps} \label{sec: tropical types} We focus on tropical maps to orthants, as these are the tropicalisations of logarithmic maps to pairs; further details can be found in \cite[Section~2.5]{AbramovichChenGrossSiebertDegeneration}. However, our arguments also apply to tropical maps to affine spaces (see Section~\ref{sec: variants}).

Consider the strictly convex rational polyhedral cone
\[ \R^n_+ \colonequals (\upsigma=\RR^n_{\geqslant 0}, N=\Z^n).\]

\begin{definition} \label{def: tropical type} A \textbf{tropical type} of map to $\RR^n_+$ consists of the following data:
\begin{enumerate}
\item A finite connected graph $\Gamma$ consisting of vertices, edges and semi-infinite legs. Each leg $l \in L(\Gamma)$ carries an index label $i_l \in [m]$ and each vertex $v \in V(\Gamma)$ carries genus and multi-degree labels
\[ g_v \in \N, \quad d_v \in \Z^n.\medskip\]
\item To every vertex, edge and leg of $\Gamma$ an associated face of $\RR^n_+$
\[ v \rightsquigarrow \upsigma_v, e \rightsquigarrow \upsigma_e, l \rightsquigarrow \upsigma_l \]
such that if $v \leqslant e$ then $\upsigma_v \leqslant \upsigma_e$ for every vertex $v$ and adjacent edge or leg $e$.\medskip
\item For every oriented edge or leg $\vec{e}$ a slope vector
\[ m_{\vec{e}} \in \Z^n\]
contained in the linear space spanned by $\upsigma_e$ and satisfying $m_{\vec{e}} = -m_{\cev{e}}$.
\end{enumerate}
These data are required to satisfy the balancing condition at each vertex: for every $v \in V(\Gamma)$
\[ d_v = \sum_{v \leqslant \vec{e}} m_{\vec{e}} \]
where the sum is over adjacent edges and legs, all oriented away from $v$. The balancing condition uniquely determines the $d_v$ in terms of the $m_{\vec{e}}$.
\end{definition}

\subsection{Tropical monoids} \label{sec: tropical monoid}
Fix a tropical type $\uptau$. There is an associated \textbf{tropical moduli cone} parametrising maps
\[ \ftrop \colon \tropC \to \R^n_+\]
of type $\uptau$. Here $\tropC$ is a metric enhancement of $\Gamma$ and the map $\ftrop$ is constrained to have slope $m_{\vec{e}}$ along every oriented edge and leg. Following \cite{GrossSiebertLog} we formulate the discussion in terms of the toric monoid dual to the tropical moduli cone.

\begin{notation} Given a set $I \subseteq [n]$ we let
\[ \upsigma_I \leqslant \R^n_+ \]
denote the face spanned by the standard basis vectors $\{e_i \suchthat i \in I\}$. For every vertex $v \in V(\Gamma)$ let $I(v) \subseteq [n]$ denote the set with $\upsigma_{I(v)} = \upsigma_v$ and similarly for edges and legs. 	
\end{notation}

\begin{definition} \label{def: tropical presentation} Given a tropical type $\uptau$, the corresponding \textbf{tropical presentation} $(G_\uptau|R_\uptau)$ is defined as follows. Generators are given by the symbols
\[ G_\uptau \colonequals \{ \ell_e \suchthat e \in E(\Gamma) \} \sqcup \{ \ftrop(v)_i \suchthat v \in V(\Gamma), i \in I(v) \}.\]
Here $\ell_e$ is the length of the finite edge $e$, and the vector $(\ftrop(v)_i \suchthat i \in I(v))$ is the position $\ftrop(v) \in \upsigma_{I(v)}$ of the vertex $v$.

The relations in $R_\uptau$ are indexed by pairs $(\vec{e},i)$ where $\vec{e}$ is an oriented edge and $i \in I(e)$. The associated relation is
\begin{equation} \label{eqn: monoid relation} \ftrop(v_2)_i = \ftrop(v_1)_i + (m_{\vec{e}})_i \, \ell_e \end{equation}
where $\vec{e}$ starts at $v_1$ and ends at $v_2$. Here $(m_{\vec{e}})_i$ is the component of the vector $m_{\vec{e}} \in \Z^n$ in the direction $i \in I(e) \subseteq [n]$. This relation imposes that $\ftrop$ has slope $m_{\vec{e}}$ along $\vec{e}$. 

Note that $(m_{\vec{e}})_i \in \Z$, whilst $\ftrop(v_2)_i, \ftrop(v_1)_i, \ell_e \in G_\uptau$ are all generators. If $(m_{\vec{e}})_i < 0$ then we rearrange the relation so that all coefficients are non-negative. If it happens that the containment $\upsigma_{v_j} \leqslant \upsigma_e$ is strict, then for $i \in I(e) \setminus I(v_j)$ we replace $\ftrop(v_j)_i$ by $0$ in \eqref{eqn: monoid relation}. 
\end{definition}

\begin{definition} \label{def: tropical monoid} The \textbf{tropical monoid} $P_\uptau$ is the toric monoid associated to the tropical presentation $(G_\uptau|R_\uptau)$, in the sense of Definition~\ref{def: toric monoid associated to presentation}.
\end{definition}

\begin{remark} \label{rmk: toric monoid vs GS monoid} The tropical monoid defined above is the sharpening of the basic monoid defined in \cite[Construction~1.16]{GrossSiebertLog}. In Gross--Siebert's construction, saturating the relations amounts to taking the torsion-free part of the monoid, while quotienting the groupification and taking the image of the prequotient amounts to integralising. Note that Gross--Siebert's basic monoid is not always sharp, though it is when the tropical type is representable (see \cite[Proposition~1.19]{GrossSiebertLog} and Lemma~\ref{lem: sharpening unnecessary} below). For Theorem~\ref{thm: main introduction} we are only interested in monoid presentations $(G|R)$ such that $\N^G/R$ is sharp, so whether or not we sharpen at the end is unimportant.
\end{remark}

\subsection{Representability} \label{sec: representability} Tropical types record the combinatorial data of a tropical map. However, not all tropical types arise in this way.

\begin{definition}
A tropical type $\uptau$ is \textbf{representable} if there exists a map
\[ \ftrop \colon \sqC \to \R^n_+\]
where $\sqC$ is a metric enhancement of $\Gamma$ with strictly positive edge lengths, and $\ftrop$ is a map with slope $m_{\vec{e}}$ along every oriented edge and leg, such that
\[ \ftrop(\operatorname{RelInt}(p)) \subseteq \operatorname{RelInt}(\upsigma_p)\]
for all polyhedra $p \in V(\Gamma) \sqcup E(\Gamma) \sqcup L(\Gamma)$.
\end{definition}

\begin{remark} It is essential to impose that all edge lengths are strictly positive and that the relative interior of $p$ is mapped to the relative interior of $\upsigma_p$. Otherwise, every tropical type would be trivially representable, by setting all edge lengths to zero and letting $\ftrop$ be the map collapsing $\sqC$ to the point $0 \in \R^n_+$.
\end{remark}

\begin{remark} Representability is referred to as smoothability in \cite[Section~3.4]{BNR2}. This notion should not be confused with realisability, which is typically a condition on individual tropical maps rather than tropical types (see e.g. \cite{Speyer,RSW2}).
\end{remark}

\begin{lemma} \label{lem: representable iff element of dual cone} A tropical type is representable if and only if there exists a monoid morphism 
\[ u \colon P_\uptau \to \RR_+\] 
which is non-zero on the generators $\ell_e$ and $\ftrop(v)_i$ appearing in Definition~\ref{def: tropical presentation}.
\end{lemma}

\begin{proof} The given presentation of $P_\uptau$ gives an inclusion of rational polyhedral cones
\[ \Hom(P_\uptau, \R_+) \hookrightarrow \prod_{e \in E(\Gamma)} \RR_+ \ell_e^\vee \times \prod_{v \in V(\Gamma)} \prod_{i \in I(v)} \RR_+ \ftrop(v)_i^\vee \]
which realises $\Hom(P_\uptau,\R_+)$ as the tropical moduli cone associated to $\uptau$. Every 
\[ u \in \Hom(P_\uptau,\R_+)\]
thus determines a metric enhancement of $\Gamma$ together with a map $\ftrop$ to $\R^n_+$. The condition that $u$ is non-zero on generators is then precisely the condition that the edge lengths are strictly positive and that $\ftrop$ maps the relative interior of every $p$ into the relative interior of the associated cone $\upsigma_p$.
\end{proof}

\subsection{Logarithmic maps and Artin fans} \label{sec: logarithmic maps} Let $(X|D)$ be a normal crossings pair. This means that, \'etale locally, the pair $(X|D)$ is isomorphic to affine space equipped with a collection of coordinate hyperplanes. The Artin fan
\[ \Acal_{X|D} \]
is a smooth toric Artin stack of pure dimension zero. There is a smooth morphism $X \to \Acal_{X|D}$ which identifies the strata of $(X|D)$ with the toric strata of $\Acal_{X|D}$. The Artin fan has an open cover by stacks of the form $\Acal^n$ where
\[ \Acal \colonequals [\Aaff^{\!1}/\Gm]. \]
There is a pushforward morphism on spaces of logarithmic maps
\[ \Log(X|D) \to \Log(\Acal_{X|D}) \]
where $\Acal_{X|D}$ carries the toric logarithmic structure. This morphism is equipped with a relative perfect obstruction theory. This identifies the virtual singularities of $\Log(X|D)$ with the actual singularities of $\Log(\Acal_{X|D})$. The latter are controlled by tropical combinatorics:

\begin{proposition} \label{prop: representability implies map to Artin fan exists} 
Let $\uptau$ be a tropical type of map to $\RR^n_+$. If $\uptau$ is representable, there exists a logarithmic map to the Artin fan
\[ \Ccal \to \Acal^n\]
whose tropicalisation has type $\uptau$. In a neighbourhood of this map, the moduli space $\Log(\Acal^n)$ has singularity type $\Spec \kfield[P_\uptau]$ where $P_\uptau$ is the tropical monoid corresponding to $\uptau$.
\end{proposition}

\begin{proof}
Let $\Gamma$ be the dual graph encoded in the tropical type $\uptau$. In the tropical presentation $(G_\uptau|R_\uptau)$ given in Definition~\ref{def: tropical presentation}, there is a natural inclusion $E(\Gamma) \subseteq G_\uptau$ as the edge length generators $\ell_e$. This gives rise to a homomorphism
\[ v \colon \N^{E(\Gamma)} \to P_\uptau.\]
By Lemma~\ref{lem: representable iff element of dual cone} and the representability assumption it follows that $v$ is sharp, i.e. we have $v^{-1}(0)=0$. Sharpness is a necessary and sufficient condition for a morphism of toric monoids to lift to a morphism of logarithmic points. Hence we may choose an arbitrary such lift
\begin{equation} \label{eqn: map of log points} (\Speck, P_\uptau) \to \left(\Speck, \N^{E(\Gamma)}\right). \end{equation}
Now take $C$ to be any algebraic curve with dual graph $\Gamma$. The logarithmic structure on the space of prestable curves produces a universal enhancement of $C$ to a logarithmic curve over the logarithmic point
\[\left(\Speck, \N^{E(\Gamma)}\right).\]
Pulling back along \eqref{eqn: map of log points} then produces a logarithmic curve
\[
\begin{tikzcd}
\Ccal \ar[d] \\
(\Speck,P_\uptau)
\end{tikzcd}
\]
which tropicalises to a family of tropical curves, which is a morphism of cone complexes
\begin{equation} \label{eqn: family of tropical curves}
\begin{tikzcd}
\sqC \ar[d] \\
\Hom(P_\uptau,\RR_{+}).	
\end{tikzcd}
\end{equation}
It is well-known (see e.g. \cite[Proposition~2.10]{AbramovichChenGrossSiebertDegeneration}) that logarithmic morphisms $\Ccal \to \Acal^n$ correspond to cone complex morphisms $\sqC \to \RR^n_+$. We construct the latter directly, working fibrewise with respect to \eqref{eqn: family of tropical curves}. By the definition of the generators of $P_\uptau$ a point
\[ u \in \Hom(P_\uptau,\R_+) \]
defines edge lengths $\ell_e \geqslant 0$ for $e \in E(\Gamma)$ and vertex positions $\ftrop(v) \in \upsigma_v$ for $v \in V(\Gamma)$. This defines the morphism $\sqC_u \to \R^n_+$ on the fibre over $u$. The relations in the definition of $P_\uptau$ ensure that this map has integral slopes along edges, which in turn ensures that these fibrewise morphisms glue to a global morphism of cone complexes. This produces the desired diagram
\[ 
\begin{tikzcd}
\sqC \ar[r] \ar[d] & \R^n_+\\
\Hom(P_\uptau,\R_+).	
\end{tikzcd}
\]
We have thus shown that there exists a logarithmic map to $\Acal^n$ with tropical type $\uptau$.

The space $\Log(\Acal^n)$ is logarithmically smooth over the trivial logarithmic point \cite[Proposition~1.6.1]{AbramovichWiseBirational}. Hence its singularities are governed by its charts, which are given by the tropical monoids $P_\uptau$ (see \cite[Definition~1.20]{GrossSiebertLog} and Remark~\ref{rmk: toric monoid vs GS monoid}).
\end{proof}

\begin{remark} More generally, in Proposition~\ref{prop: representability implies map to Artin fan exists} we may replace $\R^n_+$ by any cone complex $\Sigma$ and $\Acal^n$ by the Artin fan $\Acal_\Sigma$.
\end{remark}

\section{Universality}

\noindent Given a toric monoid $P$ we wish to construct a representable tropical type $\uptau$ with $P_\uptau=P$. The strategy is to choose a special presentation $(G|R)$ of $P$ and to construct $\uptau$ from the combinatorics of this presentation. 

We begin in Section~\ref{sec: presentation surgery} by showing that we can always find a presentation which is both ``bipartite'' and ``positive''. Then in Section~\ref{sec: main theorem} we give the main construction, building a tropical type which mirrors the structure of this special presentation. In Section~\ref{sec: singularities of Artin fan} we present applications to singularities of spaces of logarithmic maps, and finally in Section~\ref{sec: variants} we adapt our construction to deal with tropical maps to affine spaces.

\subsection{Presentation surgery} \label{sec: presentation surgery} We show that every monoid presentation can be replaced by one of a very specific form.

\begin{definition} \label{def: bipartite} A presentation $(G|R)$ is \textbf{bipartite} if there exists a partition of the generators
\[ G = G_1 \sqcup G_2 \]
such that every relation in $R$ is of the form $(w_1=w_2)$ where $w_1$ is a word in the elements of $G_1$ and $w_2$ is a word in the elements of $G_2$.
\end{definition}

\begin{definition} \label{def: positive} A presentation $(G|R)$ is \textbf{positive} if under the quotient homomorphism
\[ \uppi \colon \N^G \to \N^G/R \]
we have $\uppi(g) \neq 0$ for all $g \in G$, where by abuse of notation $g \in \N^G$ denotes the associated standard generator.	
\end{definition}

\begin{proposition} \label{prop: sanitising presentation} Every presentation $(G|R)$ can be replaced by a presentation $(G^\prime | R^\prime)$ which is bipartite and positive, and induces the same monoid:
\[ \N^{G^\prime}/R^\prime = \N^G/R.\]
\end{proposition}
\begin{proof} For each $g \in G$ introduce two symbols $g_1,g_2$ and define
\[ G_1 \colonequals \{ g_1 \suchthat g \in G \}, \quad G_2 \colonequals \{ g_2 \suchthat g \in G\}.\]
Set $G^\prime=G_1 \sqcup G_2$. For every relation in $R$ replace each appearance of $g$ on the left-hand side by $g_1$ and each appearance of $g$ on the right-hand side by $g_2$. Let $R^\prime$ be the set of such relations, together with the relations $g_1=g_2$ for all $g \in G$. The presentation $(G^\prime | R^\prime)$ is then bipartite and satisfies
\[ \N^{G^\prime}/ R^\prime = \N^G / R.\]
It remains to replace an arbitrary bipartite presentation $(G|R)$ with a bipartite and positive presentation $(G^\prime | R^\prime)$. Let
\[ G_0 = \{ g \in G \suchthat \uppi(g) = 0\} \]
and set $G^\prime = G \setminus G_0$. Define $R^\prime$ by taking each relation in $R$ and removing all appearances of elements $g \in G_0$. The presentation $(G^\prime | R^\prime)$ is still bipartite. To show that it induces the same monoid, consider the composite
\[ \N^{G^\prime} \hookrightarrow \N^G \xrightarrow{\uppi} \N^G/R.\]
Since $\uppi(g)=0$ for $g \in G_0$ it follows that the map $\N^{G^\prime} \to \N^G/R$ is surjective.

We now show that this map respects $R^\prime$. Given $(w_1,w_2) \in R^\prime$ its image in $\N^G \times \N^G$ belongs to $R$ up to adding elements of $G_0$ to both sides. Hence $w_1$ and $w_2$ map to the same element of $\N^G/R$ and so the map $\N^{G^\prime} \to \N^G/R$ respects $R^\prime$. It follows that there is a surjective homomorphism
\begin{equation} \label{eqn: map between different presentation quotients} \N^{G^\prime}/R^\prime \to \N^G/R.\end{equation}
It remains to show that it is injective. Consider $u,v \in \N^{G^\prime}$ mapped to the same point in $\N^G/R$. By \cite[Proposition~I.1.1.3(2)]{Ogus} this means there there exists a sequence
\[ p_0,\ldots,p_n \in \N^G \]
such that $p_0=u,p_n=v$ and for $i \in [n]$ either $(p_{i-1},p_{i})$ or $(p_{i},p_{i-1})$ belongs to the set
\[ R_+ = \left\{ (w_1+c,w_2+c) \suchthat (w_1,w_2) \in R, c \in \N^G \right\}.\]
For $p \in \N^G$ let $p^\prime \in \N^{G^\prime}$ denote the image under the projection. Then $p_0^\prime=u^\prime=u, p_n^\prime=v^\prime=v$ and if $(p_{i-1},p_{i}) = (w_1+c,w_2+c) \in R_+$ then
\[ (p_{i-1}^\prime,p_{i}^\prime) = (w_1^\prime + c^\prime, w_2^\prime + c^\prime) \in R^\prime_+\]
and similarly for $(p_i,p_{i-1})$. Hence the sequence $p_0^\prime,\ldots,p_n^\prime \in \N^{G^\prime}$ witnesses the identity $u=v$ in $\N^{G^\prime}/R^\prime$. We conclude that \eqref{eqn: map between different presentation quotients} is injective, and hence there is an isomorphism
\begin{equation}  \N^{G^\prime}/R^\prime \cong \N^G/R.\end{equation}
This isomorphism fits into a commuting square
\[
\begin{tikzcd}
\N^{G^\prime} \ar[r,hook] \ar[d,"\uppi^\prime"] & \N^G \ar[d,"\uppi"] \\
\N^{G^\prime}/R^\prime \ar[r,"\cong"] & \N^G/R
\end{tikzcd}
\]
from which it follows that $\uppi^\prime(g) \neq 0$ for all $g \in G^\prime$, so $(G^\prime | R^\prime)$ is positive, as required.
\end{proof}

\subsection{Main construction} \label{sec: main theorem} We come to the main result.

\begin{theorem}[Theorem~\ref{thm: main introduction}] \label{thm: main theorem} Let $P$ be an arbitrary toric monoid. Then there exists an $n \in \N$ (depending on $P$) and a representable tropical type $\uptau$ of map to $\R^n_{+}$ such that 
\[ P = P_\uptau. \]
Moreover the tropical type can be chosen with $g(\Gamma)=0$.
\end{theorem}

\begin{proof} By Proposition~\ref{prop: sanitising presentation} we may choose a bipartite and positive presentation $(G|R)$ of $P$. Since by assumption $P$ is toric, it is the toric monoid associated to $(G|R)$. By Definition~\ref{def: tropical monoid} it suffices to find a representable tropical type $\uptau$ whose associated tropical presentation $(G_\uptau|R_\uptau)$ coincides with $(G|R)$.

Since the presentation is bipartite we have $G=G_1 \sqcup G_2$, with elements of $G_1$ only appearing on the left-hand side of relations in $R$, and elements of $G_2$ only appearing on the right-hand side. Write
\[
G_1 = \{ g_1,\ldots,g_{k_1} \}, \qquad G_2 = \{ g_{k_1+1},\ldots, g_{k_1+k_2} \}
\]
and let $\Gamma$ be the following graph, with edges labeled by the generators $g \in G$
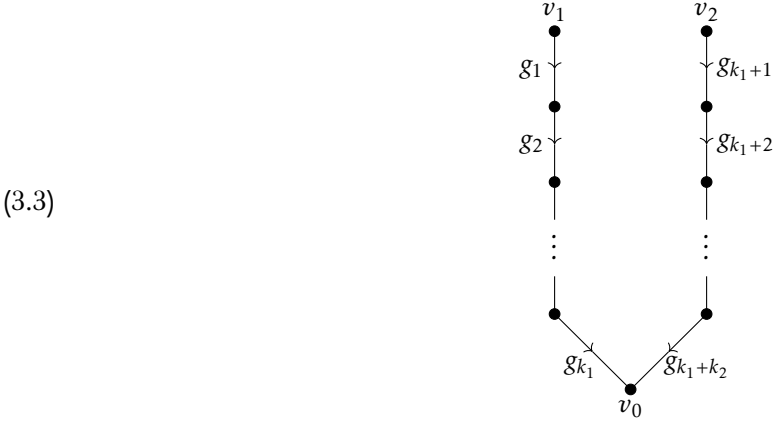
\begin{equation} \label{eqn: graph}
\begin{tikzpicture}[baseline=(current  bounding  box.center)]

\draw (-1,0) node[above]{\small$v_1$};
\draw [fill=black] (-1,0) circle[radius=2pt];
\draw (-1,0) -- (-1,-0.5) [->];
\draw (-1,-0.5) node[left]{\small$g_1$};
\draw (-1,-0.5) -- (-1,-1);
\draw [fill=black] (-1,-1) circle[radius=2pt];
\draw (-1,-1) -- (-1,-1.5) [->];
\draw (-1,-1.5) node[left]{\small$g_2$};
\draw (-1,-1.5) -- (-1,-2);
\draw [fill=black] (-1,-2) circle[radius=2pt];
\draw (-1,-2) -- (-1,-2.5);
\draw (-1,-2.75) node{$\vdots$};
\draw (-1,-3.25) -- (-1,-3.75);
\draw [fill=black] (-1,-3.75) circle[radius=2pt];
\draw [->] (-1,-3.75) -- (-0.5,-4.25);
\draw (-0.3,-4.45) node[left]{\small$g_{k_1}$};
\draw (-0.5,-4.25) -- (0,-4.75);

\draw (1,0) node[above]{\small$v_2$};
\draw [fill=black] (1,0) circle[radius=2pt];
\draw (1,0) -- (1,-0.5) [->];
\draw (1,-0.5) node[right]{\small$g_{k_1+1}$};
\draw (1,-0.5) -- (1,-1);
\draw [fill=black] (1,-1) circle[radius=2pt];
\draw (1,-1) -- (1,-1.5) [->];
\draw (1,-1.5) node[right]{\small$g_{k_1+2}$};
\draw (1,-1.5) -- (1,-2);
\draw [fill=black] (1,-2) circle[radius=2pt];
\draw (1,-2) -- (1,-2.5);
\draw (1,-2.75) node{$\vdots$};
\draw (1,-3.25) -- (1,-3.75);
\draw [fill=black] (1,-3.75) circle[radius=2pt];
\draw [->] (1,-3.75) -- (0.5,-4.25);
\draw (0.3,-4.45) node[right]{\small$g_{k_1+k_2}$};
\draw (0.5,-4.25) -- (0,-4.75);

\draw [fill=black] (0,-4.75) circle[radius=2pt];
\draw (0,-4.75) node[below]{\small$v_0$};

\end{tikzpicture}
\end{equation}
so that the path from $v_i$ to $v_0$ contains precisely the generators in $G_i$. To incorporate the relations, set $n \colonequals |R|$. We will define a tropical type of map to $\R^n_+$ with source graph $\Gamma$. For $g \in G$ let $\vec{g}$ be the corresponding edge oriented as in \eqref{eqn: graph}, and define the corresponding slope
\[ m_{\vec{g}} \in \Z^n\]
as follows: for $i \in [n]$, the component $(m_{\vec{g}})_i$ is defined to be the coefficient with which $g$ appears in the $i$th relation of $R$. Note that we always have
\[ (m_{\vec{g}})_i \geqslant 0.\]
For $v \in V(\Gamma)$ define
\[ d_v \in \Z^n\] 
as the sum of the slopes of the adjacent edges (oriented away from $v$), so that the balancing condition is satisfied. To define the cones associated to vertices, we first set
\[ \upsigma_{v_1} = \upsigma_{v_2} = 0 \leqslant \R^n_+. \]
The other cones $\upsigma_v$ are defined recursively by flowing along the oriented graph \eqref{eqn: graph}. If $\vec{g}$ is an oriented edge from $v$ to $v^\prime$ and if $I(v) \subseteq [n]$ is already determined, define
\[ I(v^\prime) \colonequals I(v) \cup \{ i \in [n] \suchthat (m_{\vec{g}})_i > 0 \}.\]
This recursively defines the indexing sets $I(v)$ and hence the cones $\upsigma_v=\upsigma_{I(v)}$. For an edge $g \in E(\Gamma)$ connecting $v$ and $v^\prime$ we declare $I(g) \colonequals I(v) \cup I(v^\prime)$. We only need to check consistency at $v_0$. We claim that whether we approach from $v_1$ or from $v_2$, we must have
\begin{equation} \label{eqn: cone of v0 is whole Rn} \upsigma_{v_0} = \R^n_+. \end{equation}
Indeed, each index $i \in [n]$ corresponds to a relation in $R$. Since the presentation is positive and the monoid is sharp, it follows that this relation must have a generator on each side with a strictly positive coefficient, guaranteeing \eqref{eqn: cone of v0 is whole Rn}. 

This completes the description of the tropical type $\uptau$. By design, the associated tropical presentation $(G_\uptau|R_\uptau)$ coincides with the initial presentation $(G|R)$ and so $P_\uptau = P$ as required.

It remains to show that $\uptau$ is representable. By Lemma~\ref{lem: representable iff element of dual cone} it is equivalent to show that there exists a monoid morphism
\[ u \colon P_\uptau \to \RR_+\] 
which is non-zero on the standard generators appearing in Definition~\ref{def: tropical presentation}. For the tropical type we have constructed, it is enough for this morphism to be non-zero on the generators $g_i$ corresponding to the edge lengths; indeed, by the structure of the cone assignments $\upsigma_v$, this implies that the relative interior of every vertex and edge is mapped into the relative interior of the corresponding cone.

By assumption the monoid $P$ is toric, so it includes into the lattice $M \colonequals P^{\gp}$ as
\[ P = \upsigma^\vee \cap M \]
where $\upsigma$ is a strictly convex rational polyhedral cone of full dimension in the dual vector space $M_{\R}^\vee$. Since $\upsigma$ has full dimension, $\upsigma^\vee$ is strictly convex, i.e. it contains the origin as a face, so there exists a normal vector $u \in M_{\R}^\vee$ with $P \subseteq \{ u \geqslant 0\}$ and $P \cap \{ u = 0 \} = 0$. Therefore $u$ restricts to a monoid morphism
\[ u \colon P \to \RR_+\]
with $u^{-1}(0)=0$. Since the presentation is positive it follows that $u(g_i) \neq 0$ for all $g_i \in G$. We conclude that $\uptau$ is representable.
\end{proof}


\begin{remark} Since $P$ is assumed toric, the presentation $(G|R)$ already gives rise to a toric monoid and so the torification step in the construction of $P_\uptau$ is redundant. However, torification plays an important and subtle role in the boundedness results of Section~\ref{sec: boundedness}.
\end{remark}

\begin{example} \label{example main construction} Consider the following strictly convex rational polyhedral cone $\upsigma^\vee \subseteq \R^2$
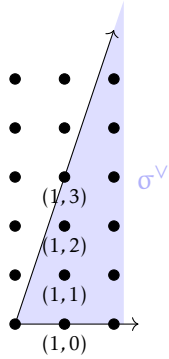
\[
\begin{tikzpicture}[scale=1.3]

\fill[blue!13!white] (2, 0.5) -- (3.1, 0.5) -- (3.1, 3.8);

\draw[lightyblue] (3.4,2) node{$\upsigma^\vee$};

\draw[black, smooth, ->] (2, 0.5) -- (3.25, 0.5);
\draw[black, smooth, ->] (2, 0.5) -- (3, 3.5);

\filldraw [black] (3, 3) circle (1.5pt);
\filldraw [black] (2.5, 3) circle (1.5pt);
\filldraw [black] (2, 3) circle (1.5pt);
\filldraw [black] (3, 2.5) circle (1.5pt);
\filldraw [black] (2.5, 2.5) circle (1.5pt);
\filldraw [black] (2, 2.5) circle (1.5pt);

\filldraw [black] (3, 2) circle (1.5pt);

\filldraw [black] (2.5, 2) circle (1.5pt);
\filldraw [black] (2, 2) circle (1.5pt);
\draw (2.5,2) node[below]{\SMALL$(1,3)$};
\filldraw [black] (3, 1.5) circle (1.5pt);
\filldraw [black] (2.5, 1.5) circle (1.5pt);
\filldraw [black] (2, 1.5) circle (1.5pt);
\filldraw [black] (3, 1) circle (1.5pt);
\filldraw [black] (2.5, 1) circle (1.5pt);
\filldraw [black] (2, 1) circle (1.5pt);

\filldraw [black] (3, 0.5) circle (1.5pt);

\filldraw [black] (2.5, 0.5) circle (1.5pt);
\draw (2.5,0.5) node[below]{\SMALL$(1,0)$};
\draw (2.5,1) node[below]{\SMALL$(1,1)$};
\draw (2.5,1.5) node[below]{\SMALL$(1,2)$};

\filldraw [black] (2, 0.5) circle (1.5pt);
\end{tikzpicture}
\]
and let $P = \upsigma^\vee \cap \Z^2$ be the monoid of lattice points. Set
\[ e_0 = (1,0), \quad e_1 = (1,1), \quad e_2 = (1,2), \quad e_3 = (1,3). \]
Then $P$ is generated by $e_0,e_1,e_2,e_3$ subject to the relations
\begin{align*} e_0 + e_2 & = 2e_1 \\
e_1 + e_3 & = 2e_2 \\
e_0 + e_3 & = e_1 + e_2.	
\end{align*}
To construct a corresponding tropical type, we first replace the above presentation by a bipartite and positive presentation. This is given by generators $G=\{ e_0,e_1,e_2,e_3,f_1,f_2 \}$ subject to the relations
\begin{align*} e_0 + f_2 & = 2e_1 \\
f_1 + e_3 & = 2e_2 \\
e_0 + e_3 & = e_1 + e_2 \\
f_1 & = e_1 \\
f_2 & = e_2
\end{align*}
so that $G_1 = \{e_0,e_3,f_1,f_2\}$ and $G_2 = \{ e_1,e_2\}$. Since there are $5$ relations we consider maps to $\R^5_+$ and produce the following tropical type, with slope vectors $m_{\vec{e}}$ indicated in blue:
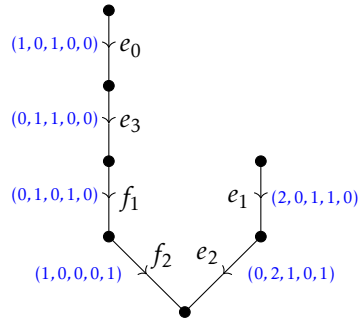
\[
\begin{tikzpicture}[baseline=(current  bounding  box.center)]

\draw [fill=black] (-1,0) circle[radius=2pt];
\draw (-1,0) -- (-1,-0.5) [->];
\draw (-1,-0.5) node[right]{\small$e_0$};
\draw [blue] (-1,-0.45) node[left]{\tiny$(1,0,1,0,0)$};
\draw (-1,-0.5) -- (-1,-1);
\draw [fill=black] (-1,-1) circle[radius=2pt];
\draw (-1,-1) -- (-1,-1.5) [->];
\draw (-1,-1.5) node[right]{\small$e_3$};
\draw [blue] (-1,-1.45) node[left]{\tiny$(0,1,1,0,0)$};
\draw (-1,-1.5) -- (-1,-2);
\draw [fill=black] (-1,-2) circle[radius=2pt];
\draw (-1,-2) -- (-1,-2.5) [->];
\draw (-1,-2.5) node[right]{\small$f_1$};
\draw [blue] (-1,-2.45) node[left]{\tiny$(0,1,0,1,0)$};
\draw (-1,-2.5) -- (-1,-3);
\draw [fill=black] (-1,-3) circle[radius=2pt];
\draw [->] (-1,-3) -- (-0.5,-3.5);
\draw (-0.3,-3.25) node[]{\small$f_{2}$};
\draw [blue] (-1.4,-3.5) node[]{\tiny$(1,0,0,0,1)$};
\draw (-0.5,-3.5) -- (0,-4);

\draw [fill=black] (0,-4) circle[radius=2pt];

\draw [fill=black] (1,-2) circle[radius=2pt];
\draw (1,-2) -- (1,-2.5) [->];
\draw (1,-2.5) node[left]{\small$e_1$};
\draw [blue] (1,-2.5) node[right]{\tiny$(2,0,1,1,0)$};
\draw (1,-2.5) -- (1,-3);
\draw [fill=black] (1,-3) circle[radius=2pt];
\draw (1,-3) -- (0.5,-3.5) [->];
\draw (0.3,-3.25) node[]{\small$e_2$};
\draw [blue] (1.4,-3.5) node[]{\tiny$(0,2,1,0,1)$};
\draw (0.5,-3.5) -- (0,-4);

\end{tikzpicture}
\]
\end{example}

\subsection{Singularities of moduli} \label{sec: singularities of Artin fan} Theorem~\ref{thm: main theorem} establishes universality for moduli spaces of tropical maps. The algebraic analogues of tropical maps are logarithmic maps to Artin fans, and so we immediately obtain the following:

\begin{theorem}[Theorem~\ref{cor: main introduction}] \label{cor: main cor} Every toric singularity appears in a moduli space $\Log(\Acal^n)$
of (genus zero) prestable logarithmic maps to the Artin fan, for some $n \in \N$ depending on the singularity.
\end{theorem}

\begin{proof} Since toric singularities are classified by toric monoids, the claim follows from Theorem~\ref{thm: main theorem} and Proposition~\ref{prop: representability implies map to Artin fan exists}.
\end{proof}


\subsection{Variations on a theme: tropical maps to affine spaces} \label{sec: variants} We now study tropical maps
\[ \sqC \to \R^n \]
where the target $\R^n$ does not carry any fan structure. These arise as tropicalisations of maps to an algebraic torus over a field with a real valuation \cite{MikhalkinTropicalCorrespondence,NishinouSiebert}.

Tropical types of such maps are defined exactly as in Definition~\ref{def: tropical type}, except that we do not include the data of cones $\upsigma_p$ associated to polyhedra $p \in V(\Gamma) \sqcup E(\Gamma) \sqcup L(\Gamma)$ and we set all $d_v=0$. The latter assumption results in the classical balancing condition
\[ \sum_{v \leqslant \vec{e}} m_{\vec{e}} = 0.\]
We consider moduli of such maps up to overall target translation. As in Definition~\ref{def: tropical monoid}, each tropical type determines a moduli cone, defined equivalently via its dual monoid as follows.

\begin{definition} \label{def: tropical monoid for map to Rn} The \textbf{tropical presentation} associated to a tropical type $\uptau$ of map to $\R^n$ has generating set indexed by the edge lengths
\[ G_\uptau \colonequals \{ \ell_e \suchthat e \in E(\Gamma) \}.\]
The relations $R_\uptau$ are defined as follows. Given an oriented cycle $\upgamma$ of edges in $\Gamma$ we consider the expression
\[ \sum_{\vec{e} \in \upgamma} m_{\vec{e}} \, \ell_e = 0. \]
Here each $m_{\vec{e}} \in \Z^n$ and so this expression gives $n$ relations amongst the generators $\ell_e$. Each such relation can be uniquely rearranged into a relation with non-negative coefficients. These constitute the set $R_\uptau$. The \textbf{tropical monoid} $P_\uptau$ is by definition the toric monoid associated to the presentation $(G_\uptau | R_\uptau)$.
\end{definition}

\begin{remark} \label{rmk: target translation} We can also consider tropical maps to $\R^n$ without identifying maps up to overall translation. In this case we simply replace the tropical monoid $P_\uptau$ by the product $P_\uptau \times \Z^n$. Since $\Speck [\Z^n] = \Gm^n$ is smooth, this does not affect the singularity type in the sense of \cite[Section~1]{VakilMurphy}.
\end{remark}

We now establish universality for tropical maps to affine spaces. Unlike in Theorem~\ref{thm: main theorem}, we cannot produce all toric monoids using only source graphs of genus zero. Indeed, if $\Gamma$ contains no cycles then Definition~\ref{def: tropical monoid for map to Rn} simply gives
\[ P_\uptau = \N^{E(\Gamma)}.\]
Nevertheless, we will see that it is sufficient to use source graphs of genus one. The proof illustrates a general principle: higher-genus continuity relations for tropical maps to affine spaces are equivalent to genus-zero continuity relations for tropical maps to orthants. This is due to the constraints on the image cones imposed in the latter case. This same principle is lurking in the discussion of monogenic types in Section~\ref{sec: monogenic types}; see in particular the proof of Proposition~\ref{prop: replace type by monogenic}.

\begin{theorem} Let $P$ be an arbitrary toric monoid. Then there exists an $n \in \N$ and a representable tropical type $\uptau$ of map to $\R^n$ such that $P=P_\uptau$. The tropical type can be chosen with $g(\Gamma)=1$.
\end{theorem}

\begin{proof} Run the construction given in the proof of Theorem~\ref{thm: main theorem}. This produces a tropical type of map to $\R^n_+$ whose tropical monoid is isomorphic to $P$ and whose source graph $\Gamma$ takes the form \eqref{eqn: graph}. Note that $g(\Gamma)=0$.

 Now glue together the leaf vertices $v_1, v_2 \in V(\Gamma)$ and forget all cones $\upsigma_p$. To remove the multi-degrees at the vertices, attach to each $v \in V(\Gamma)$ a semi-infinite leg $l_v$ and set
 \[ m_{\vec{l}_v} \colonequals - d_v.\]
Such semi-infinite legs are referred to as \textbf{$\uptau$-rays} in the literature. Finally set $d_v=0$ and observe that the balancing condition is still satisfied at $v$.
 
 This produces a tropical type $\uptau$ of map to $\R^n$ satisfying the balancing condition and with $g(\Gamma)=1$. The relations arising from the single cycle of $\Gamma$ produce the same relations as in the proof of Theorem~\ref{thm: main theorem}.
  \end{proof}

\begin{remark}[Embedded tropical curves] In contrast to tropical maps, the moduli space of embedded tropical curves (either in $\R^n_+$ or in $\R^n$) is only well-defined as a cone complex up to further subdivision. The issue is the same as that which arises when defining moduli spaces of embedded $1$-complexes \cite[Section~3]{MR20}.

As such, the universality problem is not well-posed. Every toric monoid $P$ is trivially obtained from \emph{some} moduli space of embedded tropical curves: simply take one such moduli space, resolve singularities and then subdivide so that one of the cones becomes isomorphic to the dual cone of $P$.
\end{remark}

\section{Boundedness} \label{sec: boundedness}

\noindent In Theorem~\ref{thm: main theorem} above, the dimension $n$ of the tropical target depends on the chosen monoid $P$ (it is the number of relations in a given bipartite and positive presentation of $P$). We do not know whether this dependence is essential:

\begin{question}[Question~\ref{question introduction}] \label{question main} Does there exist a single $n \in \N$ such that every toric monoid appears as the tropical monoid associated to some tropical type of map to $\R^n_+$?
\end{question}

We conclude the paper by proving that $n=1$ does not suffice. We show that while every rank-$2$ monoid does appear (Theorem~\ref{thm: get all 2D cones}), there are certain rank-$3$ monoids which do not appear (Theorem~\ref{thm: 7-gon inaccesible}).

\subsection{Monoid rank} For the rest of the paper, we consider only tropical types of maps to $\R_+$. We begin by relating the rank of the tropical monoid $P_\uptau$ to the  combinatorics of $\uptau$. The main result of this section is:
\begin{theorem} \label{thm: monoid rank} Let $\uptau$ be a representable tropical type of map to $\R_+$. Then there exists another representable tropical type $\widetilde{\uptau}$ such that 
\[ P_\uptau = P_{\widetilde{\uptau}} \oplus \N^{E_0(\Gamma)} \]
where $E_0(\Gamma) \subseteq E(\Gamma)$ is the set of edges $e$ with $m_{\vec{e}}=0$. Moreover,
\[ \rk P_{\widetilde{\uptau}}^{\gp} = |V(\widetilde{\Gamma})| - 1. \]
\end{theorem}

\subsubsection{Monogenic types} \label{sec: monogenic types} The strategy is to reduce to a special class of tropical types whose associated monoids can be easily controlled.

\begin{definition} \label{def: monogenic type} A tropical type $\uptau$ is \textbf{monogenic} if the following two conditions hold:
\begin{enumerate}
\item There is precisely one vertex $v_0 \in V(\Gamma)$ with $\upsigma_{v_0}=0$.\smallskip
\item For every other vertex $v \neq v_0$ there exists at least one adjacent edge $e \in E(\Gamma)$ with
\[m_{\vec{e}} < 0\]
where $\vec{e}$ is oriented to point away from $v$.
\end{enumerate}
\end{definition}
A monogenic type has a unique root vertex $v_0$ lying over $0 \leqslant \R_+$. All other vertices can be reached from $v_0$ along a path of rightward-sloping edges.\medskip
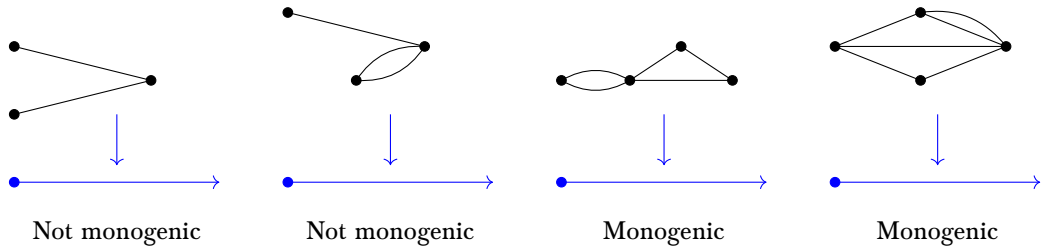
\[
\begin{tikzpicture}[scale=0.9]

\draw[blue,fill=blue] (0,0) circle[radius=2pt];
\draw[blue,->] (0,0) -- (3,0);
\draw (1.5,-0.75) node{\small{Not monogenic}};
\draw[blue,->] (1.5,1) -- (1.5,0.25);

\draw[fill=black] (0,1) circle[radius=2pt];
\draw (0,1) -- (2,1.5);
\draw[fill=black] (2,1.5) circle[radius=2pt];
\draw (2,1.5) -- (0,2);
\draw[fill=black] (0,2) circle[radius=2pt];

\draw[blue,fill=blue] (4,0) circle[radius=2pt];
\draw[blue,->] (4,0) -- (7,0);
\draw (5.5,-0.75) node{\small{Not monogenic}};
\draw[blue,->] (5.5,1) -- (5.5,0.25);

\draw[fill=black] (4,2.5) circle[radius=2pt];
\draw (4,2.5) -- (6,2);
\draw[fill=black] (6,2) circle[radius=2pt];
\draw (6,2) to[bend left] (5,1.5);
\draw (6,2) to[bend right] (5,1.5);
\draw[fill=black] (5,1.5) circle[radius=2pt];

\draw[blue,fill=blue] (8,0) circle[radius=2pt];
\draw[blue,->] (8,0) -- (11,0);
\draw (9.5,-0.75) node{\small{Monogenic}};
\draw[blue,->] (9.5,1) -- (9.5,0.25);

\draw[fill=black] (8,1.5) circle[radius=2pt];
\draw (8,1.5) to[bend left] (9,1.5);
\draw (8,1.5) to[bend right] (9,1.5);
\draw[fill=black] (9,1.5) circle[radius=2pt];
\draw (9,1.5) to (10.5,1.5);
\draw[fill=black] (10.5,1.5) circle[radius=2pt];
\draw (9,1.5) to (9.75,2);
\draw[fill=black] (9.75,2) circle[radius=2pt];
\draw (9.75,2) to (10.5,1.5);

\draw[blue,fill=blue] (12,0) circle[radius=2pt];
\draw[blue,->] (12,0) -- (15,0);
\draw (13.5,-0.75) node{\small{Monogenic}};
\draw[blue,->] (13.5,1) -- (13.5,0.25);
\draw[fill=black] (12,2) circle[radius=2pt];
\draw (12,2) to (14.5,2);
\draw[fill=black] (14.5,2) circle[radius=2pt];
\draw (12,2) to (13.25,2.5);
\draw[fill=black] (13.25,2.5) circle[radius=2pt];
\draw (13.25,2.5) to[bend left] (14.5,2);
\draw[fill=black] (13.25,1.5) circle[radius=2pt];
\draw (13.25,2.5) to(14.5,2);
\draw (12,2) to (13.25,1.5);
\draw (13.25,1.5) to (14.5,2);

\end{tikzpicture}
\]
The tropical monoid associated to a monogenic type has a presentation which is more efficient than the general-purpose presentation given in Definition~\ref{def: tropical presentation}.

\begin{definition} \label{def: monogenic presentation} Let $\uptau$ be a monogenic tropical type of map to $\R_+$. The \textbf{monogenic presentation} has generating set
\[ G_\uptau \colonequals \{ \ell_e \suchthat e \in E(\Gamma) \}.\]
The relations $R_\uptau$ are indexed by closed cycles of oriented edges in $\Gamma$. Given such a cycle $\upgamma$ we consider the following relation
\[ \sum_{\vec{e} \in \upgamma} m_{\vec{e}}\, \ell_e = 0. \]
Here each $m_{\vec{e}} \in \Z$. This expression can be uniquely rearranged to ensure that all coefficients are non-negative, giving the relation in $R_\uptau$ associated to $\upgamma$.
\end{definition}

The toric monoid associated to the monogenic presentation coincides with the tropical monoid of Definition~\ref{def: tropical monoid}. When dealing with monogenic types we will only ever use the monogenic presentation; hence we overload notation and also denote it by $(G_\uptau | R_\uptau)$.

\begin{remark} The monogenic presentation is similar to the tropical presentation associated to a tropical type of map to affine space; see Definition~\ref{def: tropical monoid for map to Rn}.\end{remark}

\begin{proposition}\label{prop: replace type by monogenic} To every tropical type $\uptau$ there is an associated monogenic type $\widetilde{\uptau}$ with
\[ P_{\uptau} = P_{\widetilde{\uptau}}.\]
If $\uptau$ is representable then so is $\widetilde{\uptau}$.
\end{proposition}

\begin{proof} Consider vertices $v \in V(\Gamma)$ which have $\upsigma_v=\R_+$ and are such that $m_{\vec{e}} \geqslant 0$ for every outward-pointing adjacent edge $\vec{e}$. These vertices violate condition (2) in Definition~\ref{def: monogenic type}.

For every such vertex, introduce a new vertex $v^\prime$ with $\upsigma_{v^\prime}=0$, and a new edge $e$ connecting $v^\prime$ to $v$ with slope $m_{\vec{e}}=1$ where $\vec{e}$ is oriented from $v^\prime$ to $v$. To satisfy the balancing condition, set $d_{v^\prime}=1$ and replace $d_v$ by $d_v-1$. 

This results in a new tropical type which has the same tropical monoid as $\uptau$. Moreover this new tropical type is guaranteed at least one vertex with $\upsigma_v=0$, and is such that all vertices $v$ with $\upsigma_v = \R_+$ have an adjacent outward-pointing edge $\vec{e}$ with $m_{\vec{e}} < 0$.

Finally, glue together all vertices with $\upsigma_v=0$. This gives a new tropical type $\widetilde{\uptau}$ which is monogenic and has the same tropical monoid as $\uptau$. Clearly $\widetilde{\uptau}$ is representable if $\uptau$ is.
\end{proof}

The following example illustrates the above process. Note that the genus of $\Gamma$ increases.

\[
\begin{tikzpicture}

\draw[blue,fill=blue] (4,0) circle[radius=2pt];
\draw[blue,->] (4,0) -- (7,0);
\draw[blue,->] (5.5,1) -- (5.5,0.25);

\draw[fill=black] (4,2.5) circle[radius=2pt];
\draw (4,2.5) -- (6,2);
\draw[fill=black] (6,2) circle[radius=2pt];
\draw (6,2) to[bend left] (5,1.5);
\draw (6,2) to[bend right] (5,1.5);
\draw[fill=black] (5,1.5) circle[radius=2pt];

\draw (7.9,1.5) node{\LARGE$\rightsquigarrow$};

\draw[blue,fill=blue] (9,0) circle[radius=2pt];
\draw[blue,->] (9,0) -- (12,0);
\draw[blue,->] (10.5,1) -- (10.5,0.25);

\draw[fill=black] (9,2.5) circle[radius=2pt];
\draw (9,2.5) -- (11,2);
\draw[fill=black] (11,2) circle[radius=2pt];
\draw (11,2) to[bend left] (10,1.5);
\draw (11,2) to[bend right] (10,1.5);
\draw[fill=black] (10,1.5) circle[radius=2pt];
\draw (10,1.5) -- (9,1.5);
\draw[fill=black] (9,1.5) circle[radius=2pt];

\draw (12.9,1.5) node{\LARGE$\rightsquigarrow$};

\draw[blue,fill=blue] (14,0) circle[radius=2pt];
\draw[blue,->] (14,0) -- (17,0);
\draw[blue,->] (15.5,1) -- (15.5,0.25);

\draw[fill=black] (14,2) circle[radius=2pt];
\draw (14,2) -- (16,2);
\draw[fill=black] (16,2) circle[radius=2pt];
\draw (16,2) to[bend left] (15,1.5);
\draw (16,2) to[bend right] (15,1.5);
\draw[fill=black] (15,1.5) circle[radius=2pt];
\draw (15,1.5) -- (14,2);

\end{tikzpicture}
\]

\subsubsection{Expansive types}
\begin{definition} A tropical type $\uptau$ is \textbf{expansive} if there are no edges $e \in E(\Gamma)$ with $m_{\vec{e}}=0$.	
\end{definition}

\begin{proposition}\label{prop: replace type by monogenic and expansive} To every tropical type $\uptau$ there is an associated monogenic and expansive type $\widetilde{\uptau}$ with
\begin{equation} \label{eqn: monoid of tau vs tautilde} P_{\uptau} = P_{\widetilde{\uptau}} \oplus \N^{E_0(\Gamma)} \end{equation}
where $E_0(\Gamma)$ is the set of edges of slope zero in the tropical type $\uptau$. If $\uptau$ is representable then so is $\widetilde{\uptau}$.
\end{proposition}

\begin{proof} First apply Proposition~\ref{prop: replace type by monogenic} to replace $\uptau$ by a monogenic type, also denoted $\uptau$, with the same tropical monoid. Then contract all edges $e \in E_0(\Gamma)$ and identify vertices of $\Gamma$ as necessary (notice that a slope zero edge may be a loop or may have distinct endpoints). This produces a monogenic and expansive type $\widetilde{\uptau}$. Since the lengths $\ell_e$ of slope zero edges $e \in E_0(\Gamma)$ are free parameters in $P_\uptau$, we immediately conclude \eqref{eqn: monoid of tau vs tautilde}. Again the representability statement is clear.
\end{proof}

\subsubsection{Monoid rank}
We will now control the rank of the tropical monoid associated to a monogenic and expansive type. Recall from Definition~\ref{def: tropical monoid} that the tropical monoid $P_\uptau$ is the torification of 
\[ \N^{G_{\uptau}}/R_\uptau.\]
Recall in addition from Remark~\ref{rmk: torification} that torification consists of three steps:
\begin{enumerate}
\item integralise and remove torsion;
\item saturate;
\item sharpen.
\end{enumerate}
Of these steps, only sharpening can change the rank of the groupification. With this in mind, the following technical lemma is necessary in order to control the monoid rank.

\begin{lemma} \label{lem: sharpening unnecessary} If $\uptau$ is a representable tropical type then the final sharpening step in the construction of $P_\uptau$ is redundant.
\end{lemma}
\begin{proof}
	
Let $P_\uptau^\flat$ denote the pre-sharpened monoid and consider the sharpening morphism
\[ w \colon P_\uptau^\flat \to P_\uptau. \]
Recall that $w$ is the quotient of $P_\uptau^\flat$ by the subgroup of units \cite[Chapter~I.1.3]{Ogus}. Assume for a contradiction that $w$ is not an isomorphism. Then there exists a non-zero $p \in P_\uptau^\flat$ with $w(p)=0$. By assumption $\uptau$ is representable, so by Lemma~\ref{lem: representable iff element of dual cone} there exists a morphism
\[ u \colon P_\uptau \to \R_+ \]
which is non-zero on the standard generators of Definition~\ref{def: tropical presentation}. Since these generate $P_\uptau^\flat$ it follows that the composite $u \circ w$ is sharp, i.e. $(u \circ w)^{-1}(0)=0$. This contradicts $w(p)=0$.

There is one caveat in the argument above: the saturation step in the construction of $P_\uptau^\flat$ can introduce additional generators. However these are all $\Q_{>0}$-linear combinations of the standard generators, so the same argument applies.
\end{proof}

\begin{remark} The above lemma essentially appears, in somewhat greater generality, as part of \cite[Proposition~1.19]{GrossSiebertLog}.
\end{remark}

\begin{theorem} \label{thm: monoid rank for monogenic and expansive} Let $\uptau$ be a representable, monogenic and expansive tropical type. Then
\[ \rk P_\uptau^{\gp} = |V(\Gamma)| -1. \]	
\end{theorem}

\begin{proof} Since $\uptau$ is representable, Lemma~\ref{lem: sharpening unnecessary} shows that the final sharpening step in the construction of $P_\uptau$ is redundant. The other torification steps --- integralising and saturating --- do not change the groupification of the monoid. As such, we may identify $P_\uptau^{\gp}$ with the torsion-free part of the groupification of the not-necessarily-toric monoid
\[ \N^{G_\uptau}/R_\uptau.\]
The torsion-free part is extracted by saturating the relations $R_\uptau$. Since groupification commutes with direct limits \cite[Chapter~I.1.3]{Ogus}, we conclude that
\[ P_\uptau^{\gp} = \Z^{G_\uptau}/R_\uptau^{\sat}.\]
Recall we are using the monogenic presentation of Definition~\ref{def: monogenic presentation}. We have $G_\uptau = E(\Gamma)$, while the relations $R_\uptau$ are indexed by cycles in $\Gamma$. Fix a spanning tree
\[ \Gamma_0 \subseteq \Gamma.\]
Let $b=b_1(\Gamma)$ denote the genus of the graph. Then $\Gamma_0$ is obtained from $\Gamma$ by deleting some edges $e_1,\ldots,e_b$. For each $i \in [b]$, the graph $\Gamma_0 \cup e_i$ contains a single closed cycle giving rise to a single relation in $R_\uptau$. Since $\uptau$ is expansive we have $m_{\vec{e}_i} \neq 0$ so that the edge length $\ell_{e_i}$ appears in this relation with non-zero coefficient.

Taken together, these relations span $R_\uptau$ over $\Z$. They are linearly independent since each relation contains a protected variable $\ell_{e_i}$ which does not appear in any other relation. Because of this independence, $R_\uptau^{\sat}$ is generated by the saturations of the generators of $R_\uptau$. We conclude that it has rank $b$, and hence
\[ \rk P_\uptau^{\gp} = \rk \left( \Z^{G_\uptau} /R_\uptau^{\sat} \right) = |E(\Gamma)| - b_1(\Gamma) = |V(\Gamma)| - 1\]
as required.
\end{proof}

\begin{proof}[Proof of Theorem~\ref{thm: monoid rank}]
Combine Proposition~\ref{prop: replace type by monogenic and expansive} and Theorem~\ref{thm: monoid rank for monogenic and expansive}.
\end{proof}

\subsection{Cautionary tale: redemption through saturation} Now consider a toric monoid $P$ with
\[ \rk P^{\gp}=2.\]
Such monoids are given by $P = \upsigma^\vee \cap M$ for $M \cong \Z^2$ a two-dimensional lattice and $\upsigma \subseteq M_{\R}^\vee$ a strictly convex rational polyhedral cone of full dimension in the dual vector space. Assume that $\upsigma$ is singular, so that $P \not\cong \N^2$.

If $\uptau$ is a representable tropical type with $P_\uptau \cong P$ then $\uptau$ must be expansive, since $P$ does not contain any $\N$ factors. By Proposition~\ref{prop: replace type by monogenic and expansive} we may therefore assume that $\uptau$ is monogenic. From Theorem~\ref{thm: monoid rank for monogenic and expansive} we conclude $|V(\Gamma)| - 1 = \rk P_\uptau^{\gp}=2$ and so
\[ |V(\Gamma)| = 3.\]
Representable, monogenic, expansive tropical types with three vertices are easily enumerated. Those giving rise to singular monoids all essentially take the form
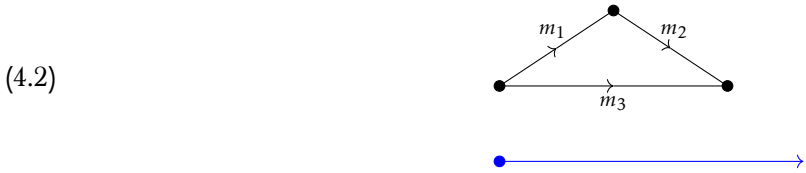
\begin{equation} \label{eqn: triangle tropical type}
\begin{tikzpicture}[baseline=(current  bounding  box.center)]

\draw[blue,fill=blue] (0,0) circle[radius=2pt];
\draw[blue,->] (0,0) -- (4,0);

\draw[fill=black] (0,1) circle[radius=2pt];
\draw[->] (0,1) -- (0.75,1.5);
\draw (0.7,1.5) node[above]{\SMALL$m_1$};
\draw (0.75,1.5) -- (1.5,2);

\draw[fill=black] (1.5,2) circle[radius=2pt];
\draw[->] (1.5,2) -- (2.25,1.5);
\draw (2.3,1.5) node[above]{\SMALL$m_2$};
\draw (2.25,1.5) -- (3,1);
\draw[fill=black] (3,1) circle[radius=2pt];

\draw[->] (0,1) -- (1.5,1);
\draw (1.5,1) node[below]{\SMALL$m_3$};
\draw (1.5,1) -- (3,1);

\end{tikzpicture}
\end{equation}
for some $m_i > 0$. (We can also have parallel edges, but the resulting monoid will be isomorphic to one associated to a tropical type with no parallel edges; see Section~\ref{sec: unparalleled monoid}.)

On initial inspection, this seems to rule out many monoids: since there are only three edges, the resulting monoid must be generated by three elements, and there are certainly monoids of rank two requiring more than three generating elements.

However, this argument overlooks a crucial technical point. The tropical monoid $P_\uptau$ is obtained as the torification of the monoid $\N^{G_\uptau}/R_\uptau$. During this process, the saturation step in particular can increase the number of generators. In fact, we have:

\begin{theorem} \label{thm: get all 2D cones} Let $P$ be a toric monoid with $\rk P^{\gp}=2$. Then there exists a representable tropical type $\uptau$ of map to $\R_+$ with $P_\uptau \cong P$. Moreover $\uptau$ may be taken to be of the form \eqref{eqn: triangle tropical type}.
\end{theorem}
\begin{proof} Following \cite[Section~2.2]{FultonToric}, $P$ is isomorphic to the monoid of lattice points $\upsigma^\vee \cap M$ for a rational polyhedral cone $\upsigma^\vee \subseteq M_{\R}$ of the form:
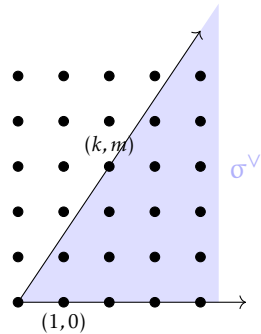
\[
\begin{tikzpicture}[scale=1.2]

\fill[blue!13!white] (2, 0.5) -- (4.2, 0.5) -- (4.2, 3.8);

\draw[lightyblue] (4.5,2) node{$\upsigma^\vee$};

\draw[black, smooth, ->] (2, 0.5) -- (4.5, 0.5);
\draw[black, smooth, ->] (2, 0.5) -- (4, 3.5);

\filldraw [black] (4, 3) circle (1.5pt);
\filldraw [black] (3.5, 3) circle (1.5pt);
\filldraw [black] (3, 3) circle (1.5pt);
\filldraw [black] (2.5, 3) circle (1.5pt);
\filldraw [black] (2, 3) circle (1.5pt);
\filldraw [black] (4, 2.5) circle (1.5pt);
\filldraw [black] (3.5, 2.5) circle (1.5pt);
\filldraw [black] (3, 2.5) circle (1.5pt);
\filldraw [black] (2.5, 2.5) circle (1.5pt);
\filldraw [black] (2, 2.5) circle (1.5pt);
\filldraw [black] (4, 2) circle (1.5pt);
\filldraw [black] (3.5, 2) circle (1.5pt);

\filldraw [black] (3, 2) circle (1.5pt);
\draw (3,2) node[above]{\SMALL$(k,m)$};

\filldraw [black] (2.5, 2) circle (1.5pt);
\filldraw [black] (2, 2) circle (1.5pt);
\filldraw [black] (4, 1.5) circle (1.5pt);
\filldraw [black] (3.5, 1.5) circle (1.5pt);
\filldraw [black] (3, 1.5) circle (1.5pt);
\filldraw [black] (2.5, 1.5) circle (1.5pt);
\filldraw [black] (2, 1.5) circle (1.5pt);
\filldraw [black] (4, 1) circle (1.5pt);
\filldraw [black] (3.5, 1) circle (1.5pt);
\filldraw [black] (3, 1) circle (1.5pt);
\filldraw [black] (2.5, 1) circle (1.5pt);
\filldraw [black] (2, 1) circle (1.5pt);

\filldraw [black] (4, 0.5) circle (1.5pt);
\filldraw [black] (3.5, 0.5) circle (1.5pt);
\filldraw [black] (3, 0.5) circle (1.5pt);

\filldraw [black] (2.5, 0.5) circle (1.5pt);
\draw (2.5,0.5) node[below]{\SMALL$(1,0)$};

\filldraw [black] (2, 0.5) circle (1.5pt);
\end{tikzpicture}
\]
This cone is generated by the primitive lattice vectors
\[ v_1 = (1,0), \qquad v_2=(k,m)\]
where $0 < k \leq m$ and $\gcd(k,m)=1$. Choose a lattice vector $v_3 \in \upsigma^\vee \cap M$ such that
\begin{equation} \label{eqn: vi generate Z2} \Z v_1 + \Z v_2 + \Z v_3 = \Z^2.\end{equation}
For example, taking $v_3 = (1,1)$ is sufficient. Consider the monoid
\[ Q \colonequals \N v_1 + \N v_2 + \N v_3 \subseteq \Z^2.\]
From $Q^{\gp} = \Z^2$ and $Q \otimes \Q_{\geqslant 0} = \upsigma^\vee \cap \Q^2$ we conclude $Q^{\sat} = P$. Now consider the surjection
\[ \Z^3 \to \Z^2 \]
given by $e_i \mapsto v_i$. The kernel gives a unique nontrivial $\Z$-linear dependence between the $v_i$. Since $v_3$ belongs to the cone generated by $v_1$ and $v_2$, this relation must take the form
\[ m_1 v_1 + m_2 v_2 = m_3 v_3\]
for some $m_i \geqslant 0$. This shows that
\[ Q = \N^3_{v_1 v_2 v_3} / (m_1 v_1 + m_2 v_2 = m_3 v_3).\]
Then the tropical type \eqref{eqn: triangle tropical type} gives rise to a monoid presentation $(G_\uptau | R_\uptau)$ with
\[ \N^{G_\uptau}/R_\uptau = Q.\]
The monoid $Q$ is integral, torsion-free and sharp. As such, torification is equivalent to saturation, giving
\[ P_{\uptau} = \left(\N^{G_\uptau}/R_\uptau\right)^{\sat}= Q^{\sat} = P \]
as required.
\end{proof}

\begin{example}\label{example get 2D cone} Recall the monoid $P$ from Example~\ref{example main construction}. There we constructed a tropical type of map to $\R^5_+$ whose tropical monoid was isomorphic to $P$. But Theorem~\ref{thm: get all 2D cones} shows that in fact $P$ can be obtained from a tropical type of map to $\R_+$ of the form \eqref{eqn: triangle tropical type}, in this case by taking $m_1=2, m_2=1, m_3=3$. Note that the saturation step in the construction of the tropical monoid is crucial.
	\end{example}

\subsection{Unparalleled monoid} \label{sec: unparalleled monoid} We see from Theorem~\ref{thm: get all 2D cones} that the minimal number of generators of $P_\uptau$ is not bounded in terms of $|E(\Gamma)|$. This is because saturation can increase the minimal number of generators. However, given a monoid
\[ Q \subseteq \Z^r \]
saturating $Q$ cannot increase the number of extremal rays of the cone $Q \otimes \R_+$. This is the key insight which leads to the proof of Theorem~\ref{thm: 7-gon inaccesible}, and which we formalise in this section.

Fix a representable, monogenic, expansive tropical type $\uptau$ and let $(G_\uptau | R_\uptau)$ be the associated monogenic presentation of Definition~\ref{def: monogenic presentation}.

\begin{definition} The \textbf{unparalleled presentation} associated to $\uptau$ is constructed from $(G_\uptau | R_\uptau)$ as follows. For every adjacent pair of vertices $v_1,v_2 \in V(\Gamma)$ define
\[ m_{v_1v_2} \colonequals \operatorname{lcm} \left\{ |m_{\vec{e}}| \suchthat e \in E(\Gamma) \text{ connecting } v_1 \text{ and } v_2 \right\} \in \N.\]
Introduce a new generator $\ell_{v_1 v_2}$ for each such pair, together with the new relations
\[ \dfrac{m_{v_1v_2}}{|m_{\vec{e}}|} \ell_{v_1v_2} = \ell_e \]
for each $e \in E(\Gamma)$ connecting $v_1$ and $v_2$. 
\end{definition}

\begin{definition} The \textbf{unparalleled monoid} $P_\uptau^{\nparallel}$ is the integralisation and torsion-free part of the monoid associated to the unparalleled presentation.
\end{definition}

\begin{lemma} \label{lem: properties of unparalleled monoid} There are natural inclusions
\[ P_\uptau^{\flat} \hookrightarrow P_{\uptau}^{\nparallel} \hookrightarrow P_\uptau \]
which induce isomorphisms on saturations. The unparalleled monoid admits a generating set of size at most
\begin{equation} \label{eqn: generators of unparalleled monoid} {|V(\Gamma)| \choose 2}.\end{equation}
\end{lemma}
\begin{proof} The natural morphism $P_\uptau^{\flat} \hookrightarrow P_{\uptau}^{\nparallel}$ is of finite index and identifies groupifications. Hence it also identifies saturations, and by definition $(P_\uptau^\flat)^{\sat}=P_\uptau$.

The monogenic presentation has $G_\uptau = E(\Gamma)$. The unparalleled presentation then reduces this generating set down to a single parameter for each pair of adjacent vertices, from which we conclude \eqref{eqn: generators of unparalleled monoid}. Note that since $\uptau$ is representable and expansive, there are no loops.
\end{proof}

\subsection{The $7$-gon is inaccessible} We are finally in a position to prove the main result of this section.

\begin{theorem}[Theorem~\ref{thm: 7-gon inaccesible introduction}] \label{thm: 7-gon inaccesible} Fix $k \geqslant 7$ and consider an arbitrary convex $k$-gon
\[ S \subseteq \Z^2 \otimes \R\] 
whose vertices are lattice points. Let $\upsigma^\vee \subseteq \Z^3 \otimes \R$ be the cone over $S \times \{ 1\}$ and let
\[P \colonequals \upsigma^\vee \cap \Z^3\]
be the associated monoid of lattice points. Then there does not exist any representable tropical type $\uptau$ of map to $\R_+$ with $P_\uptau=P$.	
\end{theorem}

\begin{proof} Suppose for a contradiction that such a tropical type $\uptau$ exists. Since $P$ contains no $\N$ factors we may assume by Proposition~\ref{prop: replace type by monogenic and expansive} that $\uptau$ is monogenic and expansive. By Theorem~\ref{thm: monoid rank for monogenic and expansive} we have $|V(\Gamma)|=4$. From this and Lemma~\ref{lem: properties of unparalleled monoid} we see that the unparalleled monoid is generated by at most $6$ elements. On the other hand we have
\[ P_\uptau^{\nparallel} \hookrightarrow P \]
which identifies saturations. Consequently, $P_\uptau^{\nparallel}$ must contain the primitive generators of the rays over the vertices of the $k$-gon $S$. Since these are extremal rays, it follows that
\[ P_\uptau^{\nparallel}\]
requires at least $k \geqslant 7$ generating elements, contradicting Lemma~\ref{lem: properties of unparalleled monoid}.
\end{proof}

\begin{corollary} With $P$ as above, the toric singularity $\Speck[P]$ does not appear in any moduli space $\Log(\Acal)$ of prestable logarithmic maps to the universal smooth pair.
\end{corollary}

\footnotesize
\bibliographystyle{alpha}
\bibliography{Bibliography.bib}\medskip

\newcommand{\etalchar}[1]{$^{#1}$}
\begin{thebibliography}{ACM{\etalchar{+}}16}

\bibitem[AC14]{AbramovichChenLog}
D.~Abramovich and Q.~Chen.
\newblock Stable logarithmic maps to {D}eligne-{F}altings pairs {II}.
\newblock {\em Asian J. Math.}, 18(3):465--488, 2014.

\bibitem[ACGS20]{AbramovichChenGrossSiebertDegeneration}
D.~Abramovich, Q.~Chen, M.~Gross, and B.~Siebert.
\newblock Decomposition of degenerate {G}romov-{W}itten invariants.
\newblock {\em Compos. Math.}, 156(10):2020--2075, 2020.

\bibitem[ACM{\etalchar{+}}16]{AbramovichEtAlSkeletons}
D.~Abramovich, Q.~Chen, S.~Marcus, M.~Ulirsch, and J.~Wise.
\newblock Skeletons and fans of logarithmic structures.
\newblock In {\em Nonarchimedean and tropical geometry}, Simons Symp., pages
  287--336. Springer, 2016.

\bibitem[AP16]{AP1}
K.~A. Adiprasito and A.~Padrol.
\newblock A universality theorem for projectively unique polytopes and a
  conjecture of {S}hephard.
\newblock {\em Israel J. Math.}, 211(1):239--255, 2016.

\bibitem[APT15]{APL}
K.~A. Adiprasito, A.~Padrol, and L.~Theran.
\newblock Universality theorems for inscribed polytopes and {D}elaunay
  triangulations.
\newblock {\em Discrete Comput. Geom.}, 54(2):412--431, 2015.

\bibitem[AW18]{AbramovichWiseBirational}
D.~Abramovich and J.~Wise.
\newblock Birational invariance in logarithmic {G}romov-{W}itten theory.
\newblock {\em Compos. Math.}, 154(3):595--620, 2018.

\bibitem[BF97]{BehrendFantechi}
K.~Behrend and B.~Fantechi.
\newblock The intrinsic normal cone.
\newblock {\em Invent. Math.}, 128(1):45--88, 1997.

\bibitem[BJLR18]{BrandtJonesLeeRanganathan}
M.~Brandt, M.~Jones, C.~Lee, and D.~Ranganathan.
\newblock Incidence geometry and universality in the tropical plane.
\newblock {\em J. Combin. Theory Ser. A}, 159:26--53, 2018.

\bibitem[BM96]{BehrendManin}
K.~Behrend and Yu. Manin.
\newblock Stacks of stable maps and {G}romov-{W}itten invariants.
\newblock {\em Duke Math. J.}, 85(1):1--60, 1996.

\bibitem[BNR21]{BNRGenus1}
L.~Battistella, N.~Nabijou, and D.~Ranganathan.
\newblock Curve counting in genus one: elliptic singularities and relative
  geometry.
\newblock {\em Algebr. Geom.}, 8(6):637--679, 2021.

\bibitem[BNR22]{BNR2}
L.~{Battistella}, N.~{Nabijou}, and D.~{Ranganathan}.
\newblock {Gromov-Witten theory via roots and logarithms}.
\newblock {\em arXiv e-prints}, March 2022.
\newblock arXiv:2203.17224. \emph{Geom. Topol.}, to appear.

\bibitem[Che14]{ChenLog}
Q.~Chen.
\newblock Stable logarithmic maps to {D}eligne-{F}altings pairs {I}.
\newblock {\em Ann. of Math. (2)}, 180(2):455--521, 2014.

\bibitem[Erm12]{ErmanHilbert}
D.~Erman.
\newblock Murphy's law for {H}ilbert function strata in the {H}ilbert scheme of
  points.
\newblock {\em Math. Res. Lett.}, 19(6):1277--1281, 2012.

\bibitem[FG10]{FantechiGoettsche}
B.~Fantechi and L.~G\"{o}ttsche.
\newblock Riemann-{R}och theorems and elliptic genus for virtually smooth
  schemes.
\newblock {\em Geom. Topol.}, 14(1):83--115, 2010.

\bibitem[Ful93]{FultonToric}
W.~Fulton.
\newblock {\em Introduction to toric varieties}, volume 131 of {\em Annals of
  Mathematics Studies}.
\newblock Princeton University Press, Princeton, NJ, 1993.
\newblock The William H. Roever Lectures in Geometry.

\bibitem[GP99]{GraberPandharipande}
T.~Graber and R.~Pandharipande.
\newblock Localization of virtual classes.
\newblock {\em Invent. Math.}, 135(2):487--518, 1999.

\bibitem[Gra19]{GraberMFOReport}
T.~Graber.
\newblock Torus localization for logarithmic stable maps.
\newblock In {\em Logarithmic enumerative geometry and mirror symmetry},
  volume~16, pages 1672--1674. 2019.
\newblock Abstracts from the workshop held June 16--22, 2019, Organized by Dan
  Abramovich, Michel van Garrel and Helge Ruddat.

\bibitem[GS13]{GrossSiebertLog}
M.~Gross and B.~Siebert.
\newblock Logarithmic {G}romov-{W}itten invariants.
\newblock {\em J. Amer. Math. Soc.}, 26(2):451--510, 2013.

\bibitem[Jel20]{JelisiejewHilbert}
J.~Jelisiejew.
\newblock Pathologies on the {H}ilbert scheme of points.
\newblock {\em Invent. Math.}, 220(2):581--610, 2020.

\bibitem[Kan23]{KannanCutPaste}
S.~Kannan.
\newblock Moduli of relative stable maps to {$\Bbb {P}^1$}: cut-and-paste
  invariants.
\newblock {\em Selecta Math. (N.S.)}, 29(4):Paper No. 54, 26, 2023.

\bibitem[KNSZ22]{KNSZ}
P.~{Kennedy-Hunt}, N.~{Nabijou}, Q.~{Shafi}, and W.~{Zheng}.
\newblock {Divisors and curves on logarithmic mapping spaces}.
\newblock {\em arXiv e-prints}, September 2022.
\newblock arXiv:2209.00630.

\bibitem[KP11]{KatzPayneRealization}
E.~Katz and S.~Payne.
\newblock Realization spaces for tropical fans.
\newblock In {\em Combinatorial aspects of commutative algebra and algebraic
  geometry}, volume~6 of {\em Abel Symp.}, pages 73--88. Springer, Berlin,
  2011.

\bibitem[LT98]{LiTian}
J.~Li and G.~Tian.
\newblock Virtual moduli cycles and {G}romov-{W}itten invariants of algebraic
  varieties.
\newblock {\em J. Amer. Math. Soc.}, 11(1):119--174, 1998.

\bibitem[LV13]{LeeVakil}
S.~H. Lee and R.~Vakil.
\newblock Mn\"{e}v--{S}turmfels universality for schemes.
\newblock In {\em A celebration of algebraic geometry}, volume~18 of {\em Clay
  Math. Proc.}, pages 457--468. Amer. Math. Soc., Providence, RI, 2013.

\bibitem[Man12]{ManolachePull}
C.~Manolache.
\newblock Virtual pull-backs.
\newblock {\em J. Algebraic Geom.}, 21(2):201--245, 2012.

\bibitem[Mik05]{MikhalkinTropicalCorrespondence}
G.~Mikhalkin.
\newblock Enumerative tropical algebraic geometry in {$\Bbb R^2$}.
\newblock {\em J. Amer. Math. Soc.}, 18(2):313--377, 2005.

\bibitem[Mn{\"e}85]{Mnev1}
N.~E. Mn{\"e}v.
\newblock Varieties of combinatorial types of projective configurations and
  convex polyhedra.
\newblock {\em Dokl. Akad. Nauk SSSR}, 283(6):1312--1314, 1985.

\bibitem[Mn{\"e}88]{Mnev2}
N.~E. Mn{\"e}v.
\newblock The universality theorems on the classification problem of
  configuration varieties and convex polytopes varieties.
\newblock In {\em Topology and geometry---{R}ohlin {S}eminar}, volume 1346 of
  {\em Lecture Notes in Math.}, pages 527--543. Springer, Berlin, 1988.

\bibitem[MR20]{MR20}
D.~{Maulik} and D.~{Ranganathan}.
\newblock {Logarithmic Donaldson-Thomas theory}.
\newblock {\em arXiv e-prints}, June 2020.
\newblock arXiv:2006.06603. \emph{Forum Math. Pi}, to appear.

\bibitem[MR21]{MolchoRanganathan}
S.~{Molcho} and D.~{Ranganathan}.
\newblock {A case study of intersections on blowups of the moduli of curves}.
\newblock {\em arXiv e-prints}, June 2021.
\newblock arXiv:2106.15194. \emph{Alg. Num. Th.}, to appear.

\bibitem[NR22]{MaxContacts}
N.~Nabijou and D.~Ranganathan.
\newblock Gromov--{W}itten theory with maximal contacts.
\newblock {\em Forum Math. Sigma}, 10:Paper No. e5, 2022.

\bibitem[NS06]{NishinouSiebert}
T.~Nishinou and B.~Siebert.
\newblock Toric degenerations of toric varieties and tropical curves.
\newblock {\em Duke Math. J.}, 135(1):1--51, 2006.

\bibitem[Ogu18]{Ogus}
A.~Ogus.
\newblock {\em Lectures on logarithmic algebraic geometry}, volume 178 of {\em
  Cambridge Studies in Advanced Mathematics}.
\newblock Cambridge University Press, Cambridge, 2018.

\bibitem[Pay08]{PayneModuliToric}
S.~Payne.
\newblock Moduli of toric vector bundles.
\newblock {\em Compos. Math.}, 144(5):1199--1213, 2008.

\bibitem[Ran17]{RanganathanSkeletons1}
D.~Ranganathan.
\newblock Skeletons of stable maps {I}: rational curves in toric varieties.
\newblock {\em J. Lond. Math. Soc. (2)}, 95(3):804--832, 2017.

\bibitem[Ran22]{RangExpansions}
D.~Ranganathan.
\newblock Logarithmic {G}romov-{W}itten theory with expansions.
\newblock {\em Algebr. Geom.}, 9(6):714--761, 2022.

\bibitem[RSPW19]{RSW2}
D.~Ranganathan, K.~Santos-Parker, and J.~Wise.
\newblock Moduli of stable maps in genus one and logarithmic geometry, {II}.
\newblock {\em Algebra Number Theory}, 13(8):1765--1805, 2019.

\bibitem[Spe14]{Speyer}
D.~E. Speyer.
\newblock Parameterizing tropical curves {I}: {C}urves of genus zero and one.
\newblock {\em Algebra Number Theory}, 8(4):963--998, 2014.

\bibitem[{Uts}20]{UtstolNodland}
B.~I. {Utst{\o}l N{\o}dland}.
\newblock {Murphy's law for toric vector bundles on smooth projective toric
  varieties}.
\newblock {\em arXiv e-prints}, February 2020.
\newblock arXiv:2002.00609.

\bibitem[Vak06]{VakilMurphy}
R.~Vakil.
\newblock Murphy's law in algebraic geometry: badly-behaved deformation spaces.
\newblock {\em Invent. Math.}, 164(3):569--590, 2006.

\end{thebibliography}

\scriptsize
\, \smallskip

\noindent Gabriel Corrigan, University of Glasgow, \href{mailto:g.corrigan.1@research.gla.ac.uk}{g.corrigan.1@research.gla.ac.uk}

\noindent Navid Nabijou, Queen Mary University of London, \href{mailto:n.nabijou@qmul.ac.uk}{n.nabijou@qmul.ac.uk}

\noindent Dan Simms, London School of Geometry and Number Theory, \href{mailto: daniel.simms.23@ucl.ac.uk}{daniel.simms.23@ucl.ac.uk}

\end{document}